\documentclass{amsart}
\usepackage{xcolor}
\newcommand{\iso}{\cong}

\newtheorem{proposition}{Proposition}[section]
\newtheorem{lemma}[proposition]{Lemma}
\newtheorem{theorem}[proposition]{Theorem}
\newtheorem{corollary}[proposition]{Corollary}

\newcommand{\st}{\colon\,}
\newcommand{\sizeof}[1]{\left\lvert{#1}\right\rvert}
\newcommand{\floor}[1]{\left\lfloor{#1}\right\rfloor}

\title[Upper bounds for inverse domination in graphs]{Upper bounds for inverse domination in graphs}
\author{Elliot Krop \and Jessica McDonald \and Gregory J.~Puleo}
\thanks{The second author is supported in part by NSF grant DMS-1600551}
\address[Jessica McDonald, Gregory J. Puleo]{Department of Mathematics and Statistics, Auburn University, Auburn, Alabama, USA 36849}
\email[Jessica McDonald]{mcdonald@auburn.edu}
\email[Gregory J. Puleo]{gjp0007@auburn.edu}
\address[Elliot Krop] {Department of Mathematics, Clayton State University,
Morrow, Georgia, USA 30260}
\email[Elliot Krop]{elliotkrop@clayton.edu}
\begin{document}

\begin{abstract}  In any graph $G$, the domination number $\gamma(G)$ is at most the independence number $\alpha(G)$. The \emph{Inverse Domination Conjecture} says that, in any isolate-free $G$, there exists pair of vertex-disjoint dominating sets $D, D'$ with $|D|=\gamma(G)$ and $|D'| \leq \alpha(G)$. Here we prove that this statement is true if the upper bound $\alpha(G)$ is replaced by $\frac{3}{2}\alpha(G) - 1$ (and $G$ is not a clique). We also prove that the conjecture holds whenever $\gamma(G)\leq 5$ or $|V(G)|\leq 16$.
\end{abstract}

\maketitle

\section{Introduction}

In this paper all graphs are simple. A \emph{dominating set} for a
graph $G$ is a set of vertices $D$ such that every vertex of $G$
either lies in $D$ or has a neighbor in $D$.  The \emph{domination
  number} of $G$, written $\gamma(G)$, is the size of a smallest
dominating set in $G$. Note that a maximum independent set is a
dominating set, so $\gamma(G) \leq \alpha(G)$, where $\alpha(G)$ is
the independence number of $G$.

If a graph $G$ has no isolates and $D$ is a minimum dominating set in
$G$, then $V(G)-D$ is also a dominating set in $G$ (owing to the
minimality of $D$); this was first observed by Ore~\cite{ore}. In
general we say that a dominating set $D'$ is an \emph{inverse
  dominating set} for a graph $G$ if there is some minimum dominating
set $D$ such that $D\cap D'=\emptyset$. A graph with isolates cannot
have an inverse dominating set, but otherwise, given Ore's
observation, we can define the \emph{inverse domination number} of a
graph $G$, written $\gamma^{-1}(G)$, as the smallest size of an
inverse dominating set in $G$. The \emph{Inverse Domination
  Conjecture} asserts that $\gamma^{-1}(G)\leq \alpha(G)$ for every
isolate-free $G$.

The Inverse Domination Conjecture originated with Kulli and
Sigarkanti~\cite{kulli}, who in fact provided an erroneous
proof. Discussion of this error and further consideration of the
conjecture first appeared in a paper of Domke, Dunbar, and
Markus~\cite{domke}. It has since been shown by Driscoll and
Krop~\cite{dk} that the weaker bound of
$\gamma^{-1}(G) \leq 2\alpha(G)$ holds in general, and Johnson, Prier
and Walsh~\cite{JPW} showed that the conjecture itself holds
whenever $\gamma(G)\leq 4$. Johnson and Walsh \cite{JW} have also
proved two fractional analogs of the conjecture, and Frendrup,
Henning, Randerath and Vestergaard \cite{FHRV} have shown that the
conjecture holds for a number of special families, including bipartite
graphs and claw-free graphs.

In this paper we prove two main results in support of the Inverse
Domination Conjecture. The first is an improvement on the $2\alpha(G)$
approximation to the conjecture.

\begin{theorem}\label{thm:three-halves}
  If $G$ is a graph with no isolated vertices and $G$ is not a clique,
  then $\gamma^{-1}(G) \leq \frac{3}{2}\alpha(G) - 1$.
\end{theorem}

Note that if $G$ is a clique and $G \neq K_1$, then trivially
$\gamma^{-1}(G) = \alpha(G) = 1$, which is why we must exclude cliques
in Theorem \ref{thm:three-halves}.

Our second main result improves the range of $\gamma(G)$ for which the conjecture is known.

\begin{theorem}\label{thm:dom5}
  If $G$ is a graph with no isolated vertices and $\gamma(G) \leq 5$, then $\gamma^{-1}(G) \leq \alpha(G)$.
\end{theorem}

As a corollary of Theorem~\ref{thm:dom5} we are also able to obtain the following.

\begin{corollary}\label{cor:v16}
  If $G$ is a graph with no isolated vertices and $\sizeof{V(G)} \leq 16$, then $\gamma^{-1}(G) \leq \alpha(G)$.
\end{corollary}

It is worth noting that Asplund, Chaffee, and Hammer \cite{ach} have
formulated a stronger form of the Inverse Domination Conjecture. In
the strengthened version one requires, for \emph{every} minimum
dominating set $D$, the existence of a dominating set $D'$ with
$D\cap D=\emptyset$ and $|D'|\leq \alpha(G)$.  It is not hard to see
that our proof for Theorem \ref{thm:three-halves} also works for this
stronger conjecture. However, the same is not true for Theorem
\ref{thm:dom5}, where we pick our minimum dominating set $D$ very
carefully.

The rest of the paper is organized as follows. In Section \ref{sec:isrs} we introduce the notion of an \emph{independent set of representatives}, or ISR, and explore the connections between ISRs and inverse domination. (In this section, we also obtain, as a corollary, the inequality $\gamma^{-1}(G) \leq b(G)$ for graphs without isolated vertices, where $b(G)$ is the largest
number of vertices in an induced bipartite subgraph of $G$.)  In Section \ref{sec:three-halves} we prove Theorem \ref{thm:three-halves}. In Section \ref{sec:dom5} we leverage the machinery of Section \ref{sec:isrs} to prove Theorem \ref{thm:dom5} and Corollary \ref{cor:v16}. 

\section{ISRs and Inverse Domination}\label{sec:isrs}

If $(X_1, \ldots, X_k)$ is a collection of sets, a \emph{set of
  representatives} for $(X_1, \ldots, X_k)$ is a set
$\{x_1, \ldots, x_k\}$ such that $x_i \in X_i$ for each $i$.
If $G$ is a graph and $V_1, \ldots, V_k$ are subsets of $V(G)$, an
\emph{independent set of representatives}, or ISR, for
$(V_1, \ldots, V_k)$ is a set of representatives for the sets
$V_1, \ldots, V_k$ that is also an independent set in $G$.
A \emph{partial ISR} for $V_1, \ldots, V_k$ is an ISR
for any subfamily of $V_1, \ldots, V_k$.

Several authors have proved various sufficient conditions guaranteeing
the existence of ISRs; many of the proofs are topological in
nature. See \cite{ABZ} for a collection of such results.  A
fundamental result on ISRs is the following sufficient condition due
to Haxell~\cite{haxell}. In what follows, given a graph $G$ and a set
$A\subseteq V(G)$, $G[A]$ denotes the subgraph of $G$ induced by
$A$. Given a collection of sets $(V_1, \ldots, V_k)$ and
$J \subseteq [k]$, we write $V_J$ for the union
$\bigcup_{j \in J}V_j$.

\begin{theorem}[Haxell~\cite{haxell}]\label{thm:haxell}
  Let $G$ be a graph and let $V_1, \ldots, V_n$ be a partition of
  $V(G)$. If, for all $S \subseteq [n]$,
  \[\gamma (G[V_S]) \geq 2|S|-1,\]
  then $G$ has an independent set $v_1, \ldots, v_n$ such that $v_i\in V_i$ for each $i$ (that is, $(V_1, \ldots, V_n)$ has an ISR).
\end{theorem}

Our basic idea for using Theorem~\ref{thm:haxell} to obtain results on
inverse domination is to apply it to a specific partition of vertices
outside $D$ (where $D$ is a minimum dominating set), namely to what
we'll call a \emph{standard partition}.

Let $G$ be a graph and suppose that
$X, Y$ are disjoint sets of vertices where $X$ dominates $Y$. The standard partition of $Y$, subject to a given ordering $(v_1, \ldots, v_n)$ of $X$, is
the partition $(V_1, \ldots, V_n)$ with
  \[V_i = N_Y(v_i) \setminus \bigcup_{j < i}V_j,\]
  where $N_Y(v_i)$ indicates those neighbors of $v_i$ that are in $Y$.
  Consider a minimum dominating set $D$, and the standard partition of
  $V(G)-D$ with respect to any ordering of $D$. If this partition has
  an ISR, then the ISR is an independent set disjoint from $D$ that
  dominates $D$. Expanding this independent set to a maximal
  independent set in $G-D$ would give an independent dominating set
  disjoint from $D$, implying that $\gamma^{-1}(G) \leq
  \alpha(G)$. However, we cannot always find an ISR for a standard
  partition of $G-D$. Instead, we obtain more technical results.

In the following, given disjoint sets
$X_1, \ldots, X_k$ and $S \subset X_1 \cup \cdots \cup X_k$, we write $i(S)$
for the set $\{j \st S\cap X_j\neq \emptyset\}$. When $S=\{v\}$, we'll denote the unique element of $i(S)$ by $i(v)$.

\begin{theorem}\label{thm:twoisrs}
  Let $G$ be a graph, let $D$ be a minimum dominating set in $G$, and
  let $F$ be a maximal independent set in $D$. Let
  $(d_1, \ldots, d_n)$ be any ordering of $D-F$, and let
  $(V_1, \ldots, V_n)$ be the standard partition of $G-D-N(F)$ subject
  to this ordering.  Then there exist two partial ISRs $R_1, R_2$ of
  $(V_1, \ldots, V_n)$ such that $i(R_1) \cap i(R_2)=\emptyset$ and $i(R_1) \cup i(R_2) = [n]$.
\end{theorem}
\begin{proof}
  Let $H$ be a graph consisting of two disjoint copies of
  $G - D - N(F)$, and let $W_1, \ldots, W_n$ be a partition of $V(H)$
  obtained by letting each $W_i$ consist of both copies of each vertex
  in $V_i$.

  We will use Theorem~\ref{thm:haxell} to obtain an ISR of $(W_1, \ldots, W_n)$.
  Let $S$ be any subset of $[n]$, and let $H'=H[W_S]$. We will show that $\gamma(H') \geq 2\sizeof{S}$.

  Observe that $H'$ consists of two disjoint copies of the subgraph $G':=G[V_S]$, so that any dominating set
  in $H'$ must dominate each of those copies. If $\gamma(H') < 2\sizeof{S}$,
  then let $C$ be a minimum dominating set of $H'$. We can partition $C$
  into $C = C_1 \cup C_2$, where $C_1$ dominates one copy of $G'$ and $C_2$
  dominates the other copy. Without loss of generality $\sizeof{C_1} \leq \sizeof{C_2}$,
  and since $\sizeof{C} < 2\sizeof{S}$, this implies $\sizeof{C_1} < \sizeof{S}$.
  Let $C'$ be the set of vertices in $G'$ corresponding to the vertices of $C_1$, and let
  $D^* = (D \setminus \{d_i \st i \in S\}) \cup C'$. We know that $D^*$ dominates $V(G)-D$, and moreover since $F \subseteq D^*$ and $F$ dominates $D-F$, we see that  $D^*$ is a dominating set of $G$.
    Since $\sizeof{D^*} < \sizeof{D}$, this contradicts the minimality of $D$.

  Thus $(W_1, \ldots, W_n)$ has some ISR $R$. We can partition $R = R_1 \cup R_2$
  where $R_1$ consists of the $R$-vertices in one copy of $G'$ and $R_2$ consists
  of the $R$-vertices in the other copy of $G'$. Now $R_1$ and $R_2$ are each independent
  subsets of $G'$, and since $R$ is an ISR we see that $i(R_1) \cap i(R_2)=\emptyset$ and $i(R_1) \cup i(R_2) = [n]$.
\end{proof}

As an immediate and useful corollary to Theorem \ref{thm:twoisrs}, we get the following.

\begin{corollary}\label{cor:ISRhalf} Let $G$ be a graph, let $D$ be a minimum dominating set in $G$, and
  let $F$ be a maximal independent set in $D$. Let
  $(d_1, \ldots, d_n)$ be any ordering of $D-F$, and let
  $(V_1, \ldots, V_n)$ be the standard partition of $G-D-N(F)$ subject
  to this ordering.  Then $(V_1, \ldots, V_n)$ has a partial ISR of size at least $n/2$.
\end{corollary}

Observe that if $D$ is a minimum dominating set in a graph $G$ without isolates, then each vertex in $D$ has a neighbor in $G-D$. These neighbors can be used to help build inverse dominating sets, and our first use of this will be in the following corollary.

\begin{corollary}
  Let $G$ be a graph without isolated vertices and let $D$ be a minimum dominating set in $G$. If  $b(G)$ is the largest
  number of vertices in an induced bipartite subgraph of $G$, then $\gamma^{-1}(D) \leq b(G)$.
\end{corollary}
\begin{proof}
  Let $F$ be a maximal independent set in $D$, and let $R_1, R_2$ be partial ISRs as in
  Theorem~\ref{thm:twoisrs}.  As $R_1$ and $R_2$ are each independent and $R_1\cap R_2 =\emptyset$,
  $R_1 \cup R_2$ induces a bipartite subgraph of $G$. Since
  $i(R_1) \cup i(R_2) = [n]$, the set $R_1 \cup R_2$ dominates $D-F$.
  Expand $R_1 \cup R_2$ to a maximal set $B \subseteq G-D$ inducing a
  bipartite subgraph.

  The maximality of $B$ implies that $B$ dominates $G-F$. Let
  $F_0 = F - N(B)$, so that $B$ dominates $G-F_0$. Observe that
  $B \cup F_0$ still induces a bipartite graph, so that
  $b(G) \geq \sizeof{B} + \sizeof{F_0}$.  On the other hand, each
  vertex $v \in F_0$ has some neighbor $v' \in V(G)-D$. Augmenting
  $B$ by adding in such a vertex $v'$ for each $v \in F_0$ yields a
  inverse dominating set of size at most
  $\sizeof{B} + \sizeof{F_0}$, which is at most $b(G)$.
\end{proof}

\section{Proof of Theorem~\ref{thm:three-halves}} \label{sec:three-halves}
\begin{theorem}\label{thm:main}
  Let $G$ be a graph, and let $D$ be a minimum dominating set in $G$. There is
  a set $T \subset V(G) - D$ such that $T$ is a dominating set in $G$ and
  $\sizeof{T} \leq \alpha(G) + \floor{\frac{\gamma(G)-1}{2}}$.
\end{theorem}
\begin{proof}
  Let $F$ be a maximal independent set in $D$, and write $D - F$ as $\{d_1, \ldots, d_n\}$.
  Let $(V_1, \ldots, V_n)$ be the standard partition of $N(D-F)$.

  Let $R$ be a largest possible partial ISR for $(V_1, \ldots,
  V_n)$. By Corollary~\ref{cor:ISRhalf}, we have
  $\sizeof{R} \geq n/2$.  Expand $R$ to a maximal independent set $S$
  in $G-D$. The set $S$ dominates every vertex of $V(G)-D$ and at
  least $n/2$ vertices of $D-F$. We now expand $S$ to dominate the
  rest of $D$.

  Let $F' = F - N(S)$. Observe that $S \cup F'$ is an independent set,
  so $\sizeof{S} + \sizeof{F'} \leq \alpha(G)$.  Expand $S$ to a set
  $S_1$ by adding an arbitrary $(G-D)$-neighbor of $v'$ for each $v' \in F'$;
  we have $\sizeof{S_1} \leq \sizeof{\alpha(G)}$. Next, expand $S_1$ to a set
  $T$ by adding an arbitrary $(G-D)$-neighbor of $w$ for each $w \in D-F-N(S_1)$;
  note that $\sizeof{D-F-N(S_1)} \leq n/2$, so $\sizeof{T} \leq \alpha(G) + n/2$.
  As $n \leq \gamma(G)-1$ and $\sizeof{T}$ is an integer, this implies that
  \[ \sizeof{T} \leq \alpha(G) + \floor{\frac{\gamma(G)-1}{2}}. \]
  Since $T$ is a dominating set in $G$, the theorem is proved.
\end{proof}
The following lemma is more general than is necessary for proving Theorem~\ref{thm:three-halves},
  but stating it in this generality will be useful for later results.
\begin{lemma}\label{lem:inddom}
  If a graph $G$ has a minimum dominating set $D$ and an independent
  set $S$ such that $S-D$ dominates $D-S$, then $\gamma^{-1}(G) \leq \alpha(G)$.
\end{lemma}
\begin{proof}
  Let $S_1 = S-D$ and let $S_2 = S \cap D$. Expand $S_1$ to a maximal independent set $S'_1$ of $G-D$.
  Now $S'_1$ dominates $G-D$. Let $S_2'$ be the set of vertices in $D$
  not dominated by $S'_1$. Observe that $S'_2 \subset S_2$, since by hypothesis
  $S_1$ dominates $D-S_2$. Hence $S'_1 \cup S'_2$ is an independent set,
  so that $\alpha(G) \geq \sizeof{S'_1} + \sizeof{S'_2}$.

  Since $D$ is a minimum dominating set of $G$ and $G$ has no isolated
  vertices, each vertex of $D$ has a neighbor outside of $D$. Let $T$
  be the vertex set obtained from $S'_1$ by adding in, for each
  $v \in S'_2$, a neighbor of $v$ outside of $D$. Now $T$ is a
  dominating set in $G$ and
  $\sizeof{T} \leq \sizeof{S'_1} + \sizeof{S'_2} \leq \alpha(G)$.
  Hence $\gamma^{-1}(G) \leq \alpha(G)$.
\end{proof}

The proof of Theorem~\ref{thm:three-halves} now follows easily. If $G$ has a minimum dominating set $D$ that is independent, then we can choose $S=D$ to vacuously meet the hypothesis of  Lemma~\ref{lem:inddom}, and hence $\gamma^{-1}(G) \leq \alpha(G) \leq (3/2)\alpha(G) - 1$.
Otherwise, $\gamma(G) \leq \alpha(G)-1$, so by Theorem~\ref{thm:main}, we have
\[ \gamma^{-1}(G) \leq \alpha(G) + \floor{\frac{\gamma(G)-1}{2}} \leq \alpha(G) + \floor{\frac{\alpha(G)-2}{2}}
  \leq \frac{3}{2}\alpha(G) - 1. \]

\section{Proof of Theorem~\ref{thm:dom5}}\label{sec:dom5}

Our proof of Theorem~\ref{thm:dom5} relies on a careful choice of
minimum dominating set. For shorthand, it will be convenient to speak
of the \emph{independence number} of a dominating set $D$ to refer to
the independence number of the induced subgraph $G[D]$, and likewise
to write $\alpha(D)$ for $\alpha(G[D])$. We will consider a dominating
set $D$ in a graph $G$ to be \emph{optimal} if it is of minimum size
and, among minimum-size dominating sets, has greatest independence
number and, subject to that, has the fewest edges in the induced
subgraph $G[D]$. In order to build inverse dominating sets in a graph
$G$, we previously used the fact that any vertex $v$ in a minimum
dominating set $D$ has a neighbor in $G-D$ (provided $G$ is
isolate-free). In some arguments, it is helpful if such a neighbor is
\emph{private} with respect to $D$; that is, if we are able to choose
$w\in V(G)-D$ with $N(w)\cap D =\{v\}$. In fact, the choice of a
private neighbor for $v$ is always possible when $D$ is a minimum
dominating set, unless $v$ is isolated in $G[D]$. The following lemma
tells us that if $D$ is optimal, we can improve on this.

\begin{lemma}\label{lem:twoprivate}
  Let $G$ be an isolate-free graph and let $D$ be an optimal dominating set in $G$. If $v \in D$
  is not an isolated vertex in $G[D]$, then $v$ has at least $2$ private
  neighbors with respect to $D$.
\end{lemma}
\begin{proof}
  Let $G_v$ be the subgraph of $G$ induced by the private neighbors of
  $v$.  We in fact show $\gamma(G_v) > 1$. Suppose to the contrary
  that $G_v$ has a dominating vertex $w$. Let
  $D' = (D - v) \cup \{w\}$. Every vertex of $G-D'$ is either $v$ itself,
  hence dominated by $w$, or a private neighbor of $v$, hence
  dominated by $w$, or a vertex of $G-D$ that is not a private neighbor
  of $D$, hence dominated by $D-v$. Thus, $D'$ is a dominating set.
  Furthermore, as $w$ was a private neighbor of $v$, the vertex $w$
  is an isolated vertex in $D'$. In particular, for any maximum independent
  set $S$ in $D$, we see that $(S - v) \cup \{w\}$ is also a maximum independent
  set in $D'$, so $D'$ has at least as large an independence number as $D$ did.
  As $w$ is isolated in $D'$ but $v$ was not isolated in $D$, we see that $D'$
    has fewer edges than $D$, contradicting the optimality of $D$.
\end{proof}

\begin{lemma}\label{lem:biglemma}
 Let $G$ be an isolate-free graph and let $D$ be an optimal dominating set in $G$. Suppose that the number of isolates in $G[D]$ is $a$. Then either $G$ has an independent set $S$ such that $S-D$
  dominates $D-S$, or all of the following are true:
  \begin{enumerate}
  \item $a+1 \leq \alpha(D) \leq \sizeof{D}-3$,
  \item $\sizeof{V(G)} + a \geq 3\sizeof{D}$, and
  \item $\sizeof{D}\geq a+5$.
  \end{enumerate}
\end{lemma}
\begin{proof}
  Assuming that $G$ has no such independent set $S$, we prove
  each part of the conclusion separately.

  (1)  If $D$ is an independent set, then taking $S=D$ gives the desired independent set. Hence $a+1 \leq \alpha(D) \leq |D|-1$, and we may choose a vertex $d^*\in D$ that is not isolated in $G[D]$. If
 $\alpha(D) = \sizeof{D}-1$, then letting $v^*$ be a private neighbor of $d^*$ and taking $S=(D-d^*)\cup \{v^*\}$ gives the desired independent set.

    Hence we may assume that $\alpha(D) = \sizeof{D}-2$. Let
    $\{d_1, \ldots, d_n\}$ be an ordering of $D$ with
    $\{d_1, \ldots, d_{n-2}\}$ independent, and let
    $(V_1, \ldots, V_n)$ be the standard partition of $N(D)$ with
    respect to this ordering.

    If there is a pair of nonadjacent vertices $v_{n-1} \in V_{n-1}$,
    $v_{n} \in V_n$, then taking
    $S = \{d_1, \ldots, d_{n-2}, v_{n-1}, v_{n}\}$ yields an
      independent set $S$ such that $S-D$ dominates $D-S$.
    Otherwise, there is a complete bipartite graph between $V_{n-1}$
    and $V_n$. Taking $v_{n-1}^*$ and $v_n^*$ to be private neighbors
    of $d_{n-1}$ and $d_n$ respectively, we see that
    $\{d_1, \ldots, d_{n-2}, v_{n-1}^*, v_n^*\}$ is a dominating set
    in $G$ having independence number $n-1$, contradicting the
    optimality of $D$.

  (2) Let $A$ be the set of $a$ isolated vertices in $D$. Notice
    that if $|N(A)|<|A|$, then $(D-A)\cup N(A)$ is a dominating set of
    size less than $D$, which is impossible. Hence, $|N(A)|\ge
    |A|$. We count $|A|$ as well as $|N(A)|$ and then apply
    Lemma~\ref{lem:twoprivate}, which implies that
    $\sizeof{V(G)} \geq \sizeof{D} + a + 2(\sizeof{D} - a) =
    3\sizeof{D} - a$.

  (3) Suppose $\sizeof{D} \leq a+4$.  By (1) we get
    $\sizeof{D}-3 \geq a+1$, so in fact $\sizeof{D}=a+4$. Moreover, by
    (1), this means that $\alpha(D)=a+1$, so
    $G[D] \iso aK_1 + K_4$.

    Write $D = \{d_1, \ldots, d_n\}$ with $d_1, \ldots, d_{n-4}$
    isolated in $G[D]$, and let $(V_1, \ldots, V_n)$ be the standard
    partition of $V(G)-D$ with respect to this ordering. Suppose first
    that there is an independent set $S_0$ in $G-D$ hitting at least
    three of the sets $\{V_{n-3}, \ldots, V_n\}$. Then define $S$ to
    be $S_0\cup\{d_1, \ldots, d_{n-4}\}$; note $S$ is independent. Out
    of the four vertices in $D-S$, at most one is not dominated by
    $S-D$. However, if such a vertex exists, then we can add it to $S$
    as well, without violating independence. Thus, we may assume that
    no such set $S_0$ exists.

    If the pair $(V_{n-3}, V_{n-2})$ is joined by a complete bipartite
    graph, then we may take $v_{n-3}^*$ and $v_{n-2}^*$ to be private
    neighbors of $d_{n-3}$ and $d_{n-2}$ respectively. Now
    \[ (D \setminus \{d_{n-3}, d_{n-2}\}) \cup \{v^*_{n-3}, v^*_{n-2}\} \]
    is a dominating set of $G$ of size $\sizeof{D}$ containing the
    independent set \[ \{d_1, \ldots, d_{n-4}, v^*_{n-3}, d_{n-1}\} \] of
    size $n-2$, contradicting the optimality of $D$.

    Otherwise, there is a pair of nonadjacent vertices
    $v_{n-3} \in V_{n-3}$ and $v_{n-2} \in V_{n-2}$.  Since, by
    assumption, this pair cannot be extended to an independent set
    that also hits one of the sets $V_{n-1}$ or $V_n$, we see that
    $\{v_{n-3}, v_{n-2}\}$ dominates $V_{n-1} \cup V_n$.  Thus
    \[\{d_1, \ldots, d_{n-4}\} \cup \{d_{n-3}, d_{n-2}, v_{n-3},
    v_{n-2}\}\] is a dominating set in $G$ containing the independent
    set \[\{d_1, \ldots, d_{n-4}, v_{n-3}, v_{n-2}\},\] contradicting
    the optimality of $D$.
\end{proof}

In the remainder of the section we will prove the inverse domination
conjecture for graphs $G$ with $\gamma(G) \leq 5$. In light of the
following lemma, it will suffice to prove the conjecture for graphs
with domination number exactly $5$.

\begin{lemma}\label{lem:leq}
  Let $k$ be a positive integer.  If $\gamma^{-1}(G) \leq \alpha(G)$
  for every isolate-free graph $G$ with $\gamma(G)=k$, then $\gamma^{-1}(G) \leq \alpha(G)$
  for every isolate-free graph $G$ with $\gamma(G) \leq k$.
\end{lemma}

\begin{proof}
  Let $G$ be an isolate-free graph with $\gamma(G) \leq k$, and let $t = k-\gamma(G)$.
  Let $G'$ be the disjoint union of $G$ and $t$ copies of $K_2$. Now $\gamma(G') = \gamma(G) + t = k$, so by hypothesis,
  $\gamma^{-1}(G') \leq \alpha(G') = \alpha(G) + t$. In particular, in $G'$ we can choose a minimum dominating set $D'$ and a second disjoint dominating set $T'$ with $\sizeof{T'} \leq \alpha(G')$.  Observe that $D'$ and $T'$ must each
  contain one vertex from every added copy of $K_2$. Hence, letting $D = D' \cap V(G)$
  and $T = T' \cap V(G)$, we see that $\sizeof{D} = \sizeof{D'} - t = \gamma(G)$
  and $\sizeof{T} \leq \alpha(G') - t = \alpha(G)$. Furthermore, $D$ and $T$
  are dominating sets in $G$. Hence, $\gamma^{-1}(G) \leq \alpha(G)$.
\end{proof}

We wish to strengthen the conclusion of Theorem~\ref{thm:twoisrs} by
eliminating the maximal independent set $F$ inside $D$, and
instead finding a pair of ISRs that jointly dominate the entire
minimum dominating set $D$. When $\gamma(G) = 5$ and
$\alpha(D) \leq 2$, we are able to do this.

\begin{lemma}\label{lem:superisrs}
  Let $D$ be an optimal dominating set in an isolate-free graph
  $G$. Suppose that $\sizeof{D} = 5$, that $\alpha(D) \leq 2$, and
  that $G[D]$ has no isolated vertices.  Then there is an ordering
  $(d_1, \ldots, d_5)$ of $D$ and a pair of independent sets $R_1$ and
  $R_2$ such that $R_1$ is an ISR for $(V_1, V_2, V_3)$ and $R_2$ is
  an ISR for $(V_4, V_5)$, where $(V_1, \ldots, V_5)$ is the
    standard partition of $G-D$ with respect to this ordering.
\end{lemma}

\begin{proof}
  Choose $d_1, d_2 \in D$ so that $\{d_1, d_2\}$ is an independent
  set, if possible. (Thus, $d_1d_2 \in E(G)$ only if $D$ is a clique.)
  Note that since $\alpha(D) \leq 2$, the set $\{d_1, d_2\}$ contains
  a maximal independent set in $D$, hence dominates
  $D - \{d_1, d_2\}$.  This implies that there are at least $3$ edges
  from $\{d_1, d_2\}$ to the rest of $D$.

  First we argue that there is an independent set $\{r_1, r_2\}$ with
  $r_i \in V_i$.  If not, then $V_1$ and $V_2$ are joined by a
  complete bipartite graph. Let $v_1^*$ and $v_2^*$ be private
  neighbors of $d_1$ and $d_2$ respectively. Observe that
  $\{v_1^*, v_2^*\} \cup (D \setminus \{d_1, d_2\})$ is a dominating set of
  $D$. Furthermore, there are no edges between $\{v_1^*, v_2^*\}$ and
  $D \setminus \{d_1, d_2\}$.  This implies that
  $\sizeof{E(D')} \leq \sizeof{E(D)} - 2$, contradicting the
  optimality of $D$. (Note that $\alpha(D') \geq \alpha(D)$
  since $\alpha(D) \leq 2$.)

  Now, since $D$ is a minimal dominating set of $G$, there is some
  vertex $r_3 \in V(G)$ not dominated by $\{d_1, d_2, r_1, r_2\}$. As
  $\{d_1, d_2\}$ dominates $D$, we have $r_3 \in V(G) - D$.  Choose
  $d_3$ to be a neighbor of $r_3$ in $D$. Let
  $R_1 = \{r_1, r_2, r_3\}$, and let $d_4$ and $d_5$ be the remaining
  vertices of $D$, ordered arbitrarily. Observe that $R_1$ is an ISR
  for $(V_1, V_2, V_3)$ in the standard partition of $V(G)-D$ with
  respect to this ordering.  It remains to find the desired $R_2$.
    
  We claim that there are nonadjacent vertices $r_4, r_5$ each with
  $r_i \in V_i$. If not, then $V_{4}$ and $V_{5}$ are joined by a
  complete bipartite graph. Let $v_{4}^*$ and $v_5^*$ be private
  neighbors of $d_4$ and $d_5$ respectively. Now
  $D' = \{d_1, d_2, d_3, v_{4}^*, v_5^*\}$ is a dominating set in
  $D$. Furthermore, since $\{d_1, d_2\}$ is a dominating set in $D$,
  there are at least two edges in the cut
  $[\{d_1, d_2, d_3\}, \{d_{4}, d_5\}]$, while by contrast there are
  no edges joining $\{v_{4}^*, v_5^*\}$ with $\{d_1, d_2,
  d_3\}$. Hence $\sizeof{E(D')} \leq \sizeof{E(D)} - 1$, contradicting the
  optimality of $D$. (Again $\alpha(D') \geq \alpha(D)$ since
    $\alpha(D) \leq 2$.)
\end{proof}
\begin{theorem}\label{thm:gamma5}
  If $G$ is an isolate-free graph with $\gamma(G) = 5$, then $G$ has a minimum dominating set $D$ such
  that $\gamma^{-1}(D) \leq \alpha(G)$.
\end{theorem}
\begin{proof}
  Let $D$ be an optimal dominating set in $G$. By Lemma
  \ref{lem:inddom} and by parts (1) and (3) of
  Lemma~\ref{lem:biglemma}, we may assume that $\alpha(D) \leq 2$ and
  that $D$ has no isolated vertices.  In particular, since $D$ is not
  an independent set, we have $\alpha(G) \geq 6$, a fact we will use
  later.

  By Lemma~\ref{lem:superisrs}, we see that there is an ordering
  $(d_1, \ldots, d_5)$ of $D$ and a pair of independent sets
  $R_1, R_2$ such that $R_1$ is an ISR for $(V_1, V_2, V_3)$ and $R_2$
  is an ISR for $(V_4, V_5)$, where $(V_1, \ldots, V_5)$ is the
  standard partition of $G-D$ for the given ordering. Among all such
  pairs $(R_1, R_2)$, choose $R_1$ and $R_2$ to minimize the number of
  edges from $R_1$ to $R_2$.

  If $(V_1, \ldots, V_5)$ has a partial ISR of size $4$, then we
  immediately get the desired conclusion: taking $R$ to be such an
  ISR, we see that $R$ dominates all of $D$ except possibly for a
    single vertex $w \in D$, so we win by letting $S = R \cup \{w\}$
    (or $S=R$) and applying Lemma~\ref{lem:inddom}.

  Thus, $(V_1, \ldots, V_5)$ has no partial ISR of size $4$, which
  implies that $R_1$ is a maximal partial ISR of this family, and
  so $R_1$ dominates $V_4 \cup V_5$.

  Let $T$ be the set of vertices in $G$ that are not dominated by $R_1 \cup R_2$.  If
  $T = \emptyset$ then we immediately have the desired conclusion, as
  $R_1 \cup R_2$ is an inverse dominating set of size $\gamma$. Thus
  we may assume that $T$ is a nonempty subset of $V(G)-D$, and in particular, $T \subseteq V_1\cup V_2\cup V_3$.

  Write $R_1 = \{r_1, r_2, r_3\}$ with $r_i \in V_i$. We claim that if
  $T$ intersects $V_j$ for some $j\in\{1,2,3\}$, then the
  corresponding vertex $r_j$ is not adjacent to any vertex of
  $R_2$. Otherwise, let $r'_j \in T\cap V_j$, and let
  $R'_1 = (R_1 \setminus \{r_j\}) \cup \{r'_j\}$.  Now $R'_1$ is an
  ISR of $(V_1, V_2, V_3)$ and, since $r'_j$ is not dominated by
  $R_1 \cup R_2$, there are fewer edges between $R'_1$ and $R_2$ than
  there were between $R_1$ and $R_2$.  This contradicts the choice of
  $R_1 \cup R_2$, establishing the claim.

  In particular, the above claim implies that $\sizeof{i(T)} = 1$, since if
  $\sizeof{i(T)} \geq 2$, then taking distinct $j, k \in i(T)$, we see
  that $R_2 \cup \{r_j, r_k\}$ is a partial ISR of
  $(V_1, \ldots, V_5)$ having size $4$, contradicting our earlier
  claim that the largest such partial ISR has size $3$.

  Let $k$ be the unique index in $i(T)$. Let
  $R^* = (R_1 \cup R_2) \setminus \{r_k\}$. We next claim that any vertex
  of $\bigcup_{j \neq k}V_j$ not dominated by $R^*$ is adjacent to all of $T$.
  Otherwise, let $v_j$ be such a vertex that is not adjacent to all of $T$,
  with $v_j \in V_j$.
  
  Let $v_k$ be a vertex of $T$ not adjacent to $v_j$. If $j \in \{1,2,3\}$, then let $R'_2 = R_2 \cup \{v_j, v_k\}$. Now
  $R'_2$ is an independent set, since $R_2 \subset R^*$ and neither $v_j$ nor $v_k$ is
  dominated by $R^*$ (by choice of $v_j$ and because $v_k\in T$). As $i(R_2) = \{4,5\}$
  this implies that $R'_2$ is a partial ISR of $(V_1, \ldots, V_5)$
  having size $4$, contradicting the earlier claim that the largest
  such ISR has size $3$. If instead $j \in \{4,5\}$, then taking
  $R'_1 = (R_1 \setminus \{r_k\}) \cup \{v_j, v_k\}$ gives the same
  contradiction.

  Hence, any vertex of $\bigcup_{j \neq k}V_j$ not dominated by $R^*$
  is adjacent to all of $T$. If there is any vertex of $\bigcup_{j \neq k}V_j$
  not dominated by $R^*$, then let $w$ be such a vertex;
  now $R_1 \cup R_2 \cup \{w\}$ is an inverse dominating set of size $6$, where
  $\alpha(G) \geq 6$, and we are done. Hence, we may assume that $R^*$
  dominates $\bigcup_{j \neq k} V_j$.

  In this case, let $D' = R^* \cup \{d_k\}$. Since $R^*$ dominates
  $\bigcup_{j \neq k} V_j$, we see that $D'$ is a dominating set of
  $G$. Since $k \leq 3$, the set $\{d_k, r_4, r_5\}$ is an independent
  set: if $d_k$ were adjacent to $r_4$, this would imply
  $r_4 \in V_k$, contradicting $r_4 \in V_4$, and likewise for $r_5$.
  This contradicts the optimality of $D$.
\end{proof}
\begin{corollary}
  If $\sizeof{V(G)} \leq 16$ then $\gamma^{-1}(G) \leq \alpha(G)$.
\end{corollary}
\begin{proof}
  Let $G$ be some graph with $\gamma^{-1}(G) > \alpha(G)$, and let $D$
  be an optimal dominating set in $G$.  Let $a$ be the number of
  isolated vertices in $G[D]$.  By Lemma~\ref{lem:inddom}, there cannot be
  any independent set $S$ such that $S-D$ dominates $D-S$, so by
  Lemma~\ref{lem:biglemma}, we have:
  \begin{enumerate}
  %\item $a+1 \leq \alpha(G[D]) \leq \sizeof{D}-3$,
  \item[(2)] $\sizeof{V(G)} + a \geq 3\sizeof{D}$, and
  \item[(3)] $\sizeof{D}\geq a+5$.
  \end{enumerate}
  By Theorem~\ref{thm:gamma5} and Lemma~\ref{lem:leq}, we have
  $\gamma(G) \geq 6$, so that $\sizeof{D} \geq 6$. If $a=0$
  then (2) yields $\sizeof{V(G)} \geq 18$. Otherwise,
  $a \geq 1$, and then (2) combined with (3) yields
  $\sizeof{V(G)} \geq 2a+15 \geq 17$.
\end{proof}
\bibliographystyle{amsplain} \bibliography{inverse}

\end{document}